\documentclass[12pt, reqno]{amsart}
\usepackage{amsmath, amsthm, amscd, amsfonts, amssymb, graphicx, color, xcolor}
\usepackage[bookmarksnumbered, colorlinks, plainpages]{hyperref}
\hypersetup{colorlinks=true,linkcolor=red, anchorcolor=green, citecolor=cyan, urlcolor=red, filecolor=magenta, pdftoolbar=true}

\newtheorem{theorem}{Theorem}[section]
\newtheorem{lemma}[theorem]{Lemma}

\newtheorem{corollary}[theorem]{Corollary}
\theoremstyle{definition}
\newtheorem{definition}[theorem]{Definition}

\numberwithin{equation}{section}

\begin{document}
	
	\setcounter{page}{1}
	
	\title[SOME GENERALIZATIONS OF K-G-FRAMES IN HILBERT\\ $C^{\ast}$- MODULE]{SOME GENERALIZATIONS OF K-G-FRAMES IN HILBERT\\ $C^{\ast}$- MODULE}
	
	\author[H. LABRIGUI, A. TOURI, S. KABBAJ]{H. LABRIGUI $^{\ast 1}$,  A. TOURI$^1$ \MakeLowercase{and} S. KABBAJ$^1$}
	
	\address{$^{1}$Department of Mathematics, University of Ibn Tofail, B.P. 133, Kenitra, Morocco}
	\email{\textcolor[rgb]{0.00,0.00,0.84}{hlabrigui75@gmail;  touri.abdo68@gmail.com; samkabbaj@yahoo.fr}}

	
	
	\date{Received: 07/12/2018; 
		\newline \indent $^{*}$Corresponding author}

\begin{abstract}
In this papers we investigate the g-frame and Bessel g-sequence related to a linear bounded operator $K$ in Hilbert $C^{\ast}$-module and we establish some results.
\end{abstract} \maketitle

\section{\textbf{Introduction}}

Frames were first introduced in 1952 by Duffin and Schaefer \cite{6} in the study of nonharmonic fourier series. Frames possess many nice properties which make them very useful in wavelet analysis, irregular sampling theory, signal processing and many other fields. The theory of frames has been generalized rapidly and various generalizations of frames in Hilbert spaces and Hilbert $C^{\ast}$-modules, for example $\ast$-K-g frames in $C^{\ast}$-module \cite{RK2}.  

In this article, we characterize the concept of a canonical $K$-dual Bessel sequence of a $K$-g-frame that generalizes the classical dual of a g-frame

The paper is organized as follows, in section 2 we briefly recall the definitions and basic properties of K-g-frames in Hilbert $C^{\ast}$-modules. In section 3, we characterize some result for K-dual Bessel g-sequence for given K-g-frames.In section 4, we use a family of linear operators to characterize atomic systems

\section{\textbf{Preliminaries}}
We begin this section with the following definition and some result.

\begin{definition}\textcolor{white}{.}
	
	Let $K\in End_{\mathcal{A}}^{\ast}(H)$, we call a sequence $\{\Lambda_{i}\}_{i\in I}$ a $K$-g-frame for $H$ with respect to $\{H_{i}\}_{i\in I}$ if there are two positive constants $A$ and $B$ such that: 
	\begin{equation*}
	A\langle K^{\ast}f,K^{\ast}f \rangle \leq \sum_{i\in I} \langle \Lambda_{i}f,\Lambda_{i}f \rangle \leq B\langle f,f\rangle, \textcolor{white}{..........} \forall f \in H 
	\end{equation*}
\end{definition}	
\begin{lemma}[see \cite{14}]\textcolor{white}{.}
	
	Let $\{\Lambda_{i}\}_{i\in I}$  be a Bessel g-sequence for $H$ with bound $B$. Then for each sequence $\{f_{i}\}_{i\in I} \in \oplus_{i\in I} H_{i}$, the series $\sum_{i\in I}\Lambda^{\ast}_{i}(f_{i})$ converges unconditionally. 	
\end{lemma}

In \cite{17}, Sum showed that every g-frame can be considered as a frame. More precisely, let $\{\Lambda_{i}\}_{i\in I}$ be a g-frame for $H$ and let $\{ e_{i,j}\}_{j\in J_{i}}$ be an orthonormal basis for $H_{i}$. Then there exists a frame $\{ u_{i,j}\}_{i\in I, j\in J_{i}}$ of $H$ such that 
\begin{equation}
u_{i,j} = \Lambda^{\ast}_{i}e_{i,j}
\end{equation}
and
\begin{equation*}
\Lambda_{i}f=\sum_{j\in J_{i}}\langle f,u_{i,j}\rangle e_{i,j},\textcolor{white}{..........}\forall f \in H
\end{equation*}
and
\begin{equation*}
\Lambda_{i}g=\sum_{j\in J_{i}}\langle g,e_{i,j}\rangle u_{i,j},\textcolor{white}{..........}\forall g \in H
\end{equation*}
We call $\{ u_{i,j}\}_{i\in I, j\in J_{i}}$ the frame induced by $\{\Lambda_{i}\}_{i\in I}$ with respect to $\{ e_{i,j}\}_{j\in J_{i}}$

The next lemma is a characterization of g-frame by a frame. 

\begin{lemma}[see \cite{17}]\textcolor{white}{.}
	
	Let $\{\Lambda_{i}\}_{i\in I}$  be a family of linear operators, and let $u_{i,j}$ be defined as in (2.1). Then $\{\Lambda_{i}\}_{i\in I}$ is a g-frame for $H$ if and only if $\{u_{i,j}\}_{i\in I, j\in J_{i}}$ 	is a frame for $H$.
\end{lemma}

\begin{lemma}[see \cite{1}]\textcolor{white}{.}
	
	Let $H$ be a separable hilbert space and $ K\in End_{\mathcal{A}}^{\ast}(H)$.Then a g-sequence $\{\Lambda_{i}\}_{i\in I}$ is a $K$-g-frame for $H$ if and only if $\{\Lambda_{i}\}_{i\in I}$ is a Bessel g-sequence  for $H$ and the range of synthesis operators $\mathcal{R}(K) \subset \mathcal{R}(T_{\Lambda})$. 
\end{lemma}

\begin{lemma}[see \cite{1}]\textcolor{white}{.}
	
	Let $H$ be a separable hilbert space and $ K\in End_{\mathcal{A}}^{\ast}(H)$. Let $\{\Lambda_{i}\}_{i\in I}$ be a family of linear operators. The following statements are equivalent :
	
	\begin{itemize}
		\item [1)] $\{\Lambda_{i}\}_{i\in I}$ is a $K$-g-frame for $H$ with respect to $\{H_{i}\}_{i\in I}$
		\item [2)] $\{\Lambda_{i}\}_{i\in I}$ is a Bessel g-sequence for $H$ and there exists a Bessel g-sequence $\{\Gamma_{i}\}_{i\in I}$ for $H$ respect to $\{H_{i}\}_{i\in I}$ such that  
		\begin{equation*}
		Kf=\sum_{i\in I}\Lambda^{\ast}_{i}\Gamma_{i}f\textcolor{white}{..........}\forall f \in H
		\end{equation*}   
	\end{itemize}
\end{lemma}

\begin{lemma}[see \cite{5}]\textcolor{white}{.}

	Let $U,V \in End_{\mathcal{A}}^{\ast}(H)$. The following statements are equivalent :
	\begin{itemize}
		\item [1)] $\mathcal{R}(U) \subset \mathcal{R}(V)$. 
		\item [2)] $UU^{\ast} \leq \lambda VV^{\ast} $ for some $\lambda \geq 0$
		\item [3)] There exists $Q\in End_{\mathcal{A}}^{\ast}(H)$ such that $U=VQ$
	\end{itemize}
	Moreover, if 1), 2) et 3) are valid, then there exists a unique operator $Q$ such that 
	\begin{itemize}
		\item [1)]$\|Q\|^{2}=inf\{\mu : UU^{\ast} \leq \lambda VV^{\ast} \}$
		\item [2)]$\mathcal{N}(U) = \mathcal{N}(C)$
		\item [3)]$\mathcal{R}(C) \subset \overline{\mathcal{R}(V^{\ast})}$
	\end{itemize}
	
\end{lemma}

\section{\textbf{K-DUAL BESSEL G-SEQUENCE FOR GIVEN K-G-FRAMES}}

\begin{definition}\textcolor{white}{.}
	
Let $K\in End_{\mathcal{A}}^{\ast}(H)$ and $\{\Lambda_{i}\}_{i\in I}$ is a K-g-frame for $H$. A Bessel g-sequence  $\{\Gamma_{i}\}_{i\in I}$ for $H$ is called a $K$-dual Bessel g-sequence of $\{\Lambda_{i}\}_{i\in I}$ if:
\begin{equation*}
 Kf=\sum_{i\in I}\Lambda^{\ast}_{i}\Gamma_{i}f\textcolor{white}{..........}\forall f \in H.
\end{equation*}  
\end{definition}

\begin{theorem}\textcolor{white}{.}
	
Let $K\in End_{\mathcal{A}}^{\ast}(H)$ and $\{\Lambda_{i}\}_{i\in I}$ is a K-g-frame for $H$ with optimal lower frame bound $A$ 

If $\Gamma = \{\Gamma_{i}\}_{i\in I}$ is a $K$-dual bessel sequence of $\{\Lambda_{i}\}_{i\in I}$, then $A \leq \|T_{\Gamma}\|^{2}$, where $T_{\Gamma}$ denotes the synthesis operator of $\Gamma$.

Moreover, there exists a unique $K$-dual bessel sequence  $\theta = \{\theta_{i}\}_{i\in I}$ of $\{\Lambda_{i}\}_{i\in I}$ such that $\|T_{\theta}\|^{2} = A$ where $T_{\theta}$ denotes the synthesis operator of $\theta$.

\end{theorem}
\begin{proof}\textcolor{white}{.}
	
Suppose that $C\geq 0$ is a lower $K$-g-frame bound of $\{\Lambda_{i}\}_{i\in I}$; then for any $f\in H$, we have : 
\begin{equation*}
C\langle K^{\ast}f,K^{\ast}f\rangle \leq \sum_{i\in I}\langle \Lambda_{i}f,\Lambda_{i}f\rangle
\end{equation*} 

So,
\begin{equation*}
C\langle K^{\ast}f,K^{\ast}f\rangle \leq \langle T_{\Lambda}f, T_{\Lambda}f\rangle
\end{equation*} 
This implies that : 
\begin{equation*}
\langle K^{\ast}f,K^{\ast}f\rangle \leq \frac{1}{C}\langle T_{\Lambda}f, T_{\Lambda}f\rangle
\end{equation*} 
So, 
\begin{align*}
A &= max\{\lambda \geq 0 , \lambda \langle K^{\ast}f,K^{\ast}f\rangle \leq \langle T_{\Lambda}f, T_{\Lambda}f\rangle, \forall f \in H \}\\
&=inf\{\mu \geq 0 , \langle K^{\ast}f,K^{\ast}f\rangle \leq \mu \langle T_{\Lambda}f, T_{\Lambda}f\rangle, \forall f \in H \}\\
\end{align*}
Since $\{\Gamma_{i}\}_{i\in I}$ is a $K$-dual Bessel g-sequence of $\{\Lambda_{i}\}_{i\in I}$,\\
for any $f\in H$, we have :
\begin{equation*}
Kf= \sum_{i\in I}\Lambda^{\ast}_{i}\Gamma_{i}f=T_{\Lambda}T^{\ast}_{\Gamma}f
\end{equation*} 
So, $K=T_{\Lambda}T^{\ast}_{\Gamma}$

Thus : $KK^{\ast}=T_{\Lambda}T^{\ast}_{\Gamma}T_{\Gamma}T^{\ast}_{\Lambda}\leq \|T_{\Gamma}\|^{2}T_{\Lambda}T^{\ast}_{\Lambda}$\\
So for any $f \in H$ we have 
\begin{equation*}
\|K^{\ast}f\|^{2}=\langle K^{\ast}f,K^{\ast}f\rangle = \langle KK^{\ast}f,f\rangle \leq \|T_{\Gamma}\|^{2}\langle T_{\Lambda}T^{\ast}_{\Lambda}f,f\rangle =\|T_{\Gamma}\|^{2}\|T_{\Lambda}f\|^{2} 
\end{equation*} 

So $\|T_{\Lambda}\|^{2} \geq A$.\\
Since $\{\Lambda_{i}\}_{i\in I}$ is a $K$-g-frame , we have that $Range(K) \subset Range(T_{\Lambda})$
By lemma.2.6 there exists a unique bounded operator $\Phi : \oplus_{i\in I}H_{i} \to H$ such that $K^{\ast} = \Phi T^{\ast}_{\Lambda}$ and 
\begin{equation*}
\|\Phi \|^{2} = inf \{\mu : \|K^{\ast}f\|^{2}\leq \mu \|T_{\Lambda}f\|^{2}, \forall f\in H\}=A
\end{equation*} 
Let $\Theta^{\ast}_{i}e_{ij}= \Phi (e_{ij}\delta_{i})$, then it is easy to check that $\Theta = \{\Theta_{i}\}_{i\in I}$ is a Bessel g-sequence, since for any $f\in H$ we have :
\begin{equation*}
K^{\ast}f=\Phi T^{\ast}_{\Lambda}f=\sum_{i\in I}\Gamma^{\ast}_{i}\Lambda_{i}f 
\end{equation*}
\end{proof}
\begin{theorem}\textcolor{white}{.}
	
Let $\{\Lambda_{i}\}_{i\in I}$ be a Bessel g-sequence for $H$ with a frame operator $S_{\Lambda}$. If $\{\Lambda_{i}\}_{i\in I}$ has a dual g-frame on $R(K)$ and  $S_{\Lambda}(R(K)) \subset R(K)$, then it is a $K$-g-frame in $H$.
\end{theorem}
\begin{proof}\textcolor{white}{.}
	
Assume that $\{\Gamma_{i}\}_{i\in I}$ is a dual g-frame of $\{\Lambda_{i}\}_{i\in I}$ on $R(K)$. Then for each $f\in H$ can be expressed as $f = f_{1} + f_{2}$, where $f_{1} \in R(K)$ et $f_{2} \in (R(K))^{\perp}$ then 
\begin{align*}
\sum_{i\in I}\langle \Lambda_{i}f,\Lambda_{i}f\rangle &= \sum_{i\in I}\langle \Lambda_{i}(f_{1} + f_{2}),\Lambda_{i}(f_{1} + f_{2})\rangle \\
&=\sum_{i\in I}\langle \Lambda_{i}f_{1},\Lambda_{i}f_{1}\rangle + \sum_{i\in I}\langle \Lambda_{i}f_{2},\Lambda_{i}f_{2}\rangle +2Re( \sum_{i\in I}\langle \Lambda^{\ast}_{i}\Lambda_{i}f_{1},f_{2}\rangle)  \\
\end{align*}
Note that 
\begin{equation*}
\sum_{i\in I} \Lambda^{\ast}_{i}\Lambda_{i}f_{1} = S_{\Lambda}f_{1} \in S_{\Lambda}(R(K)) \subset R(K)
\end{equation*}
and so we have $\sum_{i\in I}\langle \Lambda^{\ast}_{i}\Lambda_{i}f_{1},f_{2}\rangle = 0$
Hence 
\begin{equation*}
\sum_{i\in I}\langle \Lambda_{i}f,\Lambda_{i}f\rangle = \sum_{i\in I}\langle \Lambda_{i}f_{1},\Lambda_{i}f_{1}\rangle + \sum_{i\in I}\langle \Lambda_{i}f_{2},\Lambda_{i}f_{2}\rangle 
\end{equation*}
by lemma 2.2 $\sum_{i\in I}\Gamma^{\ast}_{i}\Lambda_{i}f_{1}$ converge and so does $\sum_{i\in I}\pi_{R(K)}\Gamma^{\ast}_{i}\Lambda_{i}f $ 
where $\pi_{R(K)}$ is an orthogonal projection of $H$ onto $R(K)$.
Then for each $g\in R(K)$ we have 
\begin{equation*}
\sum_{i\in I}\langle g,f_{1}\rangle = \sum_{i\in I}\langle \Lambda^{\ast}_{i}\Gamma_{i}g,f_{1}\rangle = \sum_{i\in I}\langle g, \Gamma^{\ast}_{i}\Lambda_{i}g\rangle = \langle g,\sum_{i\in I}\pi_{R(K)} \Gamma^{\ast}_{i}\Lambda_{i}g\rangle
\end{equation*}
It follows that : $f_{1} = \sum_{i\in I}\pi_{R(K)} \Gamma^{\ast}_{i}\Lambda_{i}g$

Thus

\begin{align*}
\|K^{\ast}f\|^{4}&=\|K^{\ast}(f_{1}+f_{2})\|^{4}=\|K^{\ast}f_{1}\|^{4}\rangle=\|\langle K^{\ast}f_{1},K^{\ast}f_{1}\rangle\|^{2} = \|\langle f_{1},KK^{\ast}f_{1}\rangle\|^{2}\\
&=\|\langle \sum_{i\in I}\pi_{R(K)} \Gamma^{\ast}_{i}\Lambda_{i}f_{1},KK^{\ast}f_{1}\rangle\|^{2}\\
&=\|\sum_{i\in I}\langle \Lambda_{i}f_{1},\pi_{R(K)} \Gamma_{i}KK^{\ast}f_{1}\rangle\|^{2}\\
&\leq (\sum_{i\in I}\|\Lambda_{i}f_{1}\|^{2})(\sum_{i\in I}\|\Gamma\pi_{R(K)}KK^{\ast}f_{1}\|^{2})\\
&\leq D\|K\|^{2}\|K^{\ast}f\|^{2}(\sum_{i\in I}\|\Lambda_{i}f_{1}\|^{2})
\end{align*}
Where $D$ is the Bessel bound of $\{\Gamma_{i}\}_{i\in I}$, then we have 
\begin{equation*}
D^{-1}\|K\|^{-2}\|K^{\ast}f\|^{2}\leq \sum_{i\in I}\|\Lambda_{i}f_{1}\|^{2}
\end{equation*}
Hence
\begin{align*}
\sum_{i\in I}\|\Lambda_{i}f\|^{2} &=\sum_{i\in I}\|\Lambda_{i}f_{1}\|^{2} + \sum_{i\in I}\|\Lambda_{i}f_{2}\|^{2}\\
&\geq \sum_{i\in I}\|\Lambda_{i}f_{1}\|^{2} \geq D^{-1}\|K\|^{-2}\|K^{\ast}f\|^{2}
\end{align*}
\end{proof}
\section{\textbf{Atomic system in Hilbert $C^{\ast}$-module}}

\begin{definition}[see \cite{7}]\textcolor{white}{.}
	
Let $K\in End_{\mathcal{A}}^{\ast}(H)$, a sequence $\{f_{i}\}_{i\in I}$ in $H$ is called an atomic system for $K$, it the following conditions are satisfied :
\begin{itemize}
	\item [1)] $\{f_{i}\}_{i\in I}$ is a Bessel sequence.
	\item [2)] There exists $c\geq 0$ such that for every $f\in H$,  there exists $a=\{a_{i}\}_{i\in I} \in l^{2}(I)$ such that $\|a\|_{l^{2}} \leq c\|f\|$ and $Kf=\sum_{i\in I}a_{i}f_{i}$.
\end{itemize}
\end{definition}
The following lemma characterizes an atomic system in terms of a $K$-frame.

\begin{lemma}[see \cite{7}]\textcolor{white}{.}
	
Let $\{f_{i}\}_{i\in I}$ be a sequence in $H$, and let $K\in End_{\mathcal{A}}^{\ast}(H)$, then $\{f_{i}\}_{i\in I}$ is an atomic system for $K$ if and only $\{f_{i}\}_{i\in I}$ is a $K$-frame for $H$.
\end{lemma}
 We now give a characterization of an atomic system with a sequence of linear operators.
 
 \begin{theorem}\textcolor{white}{.}
 	
 Let $\{\Lambda_{i}\}_{i\in I}$ be a family of linear operator, then the following statements are equivalent :
  \begin{itemize}
  	\item [1)] $\{\Lambda_{i}\}_{i\in I}$ is an atomic system for $K$.
  	\item [2)] $\{\Lambda_{i}\}_{i\in I}$ is a $K$-g-frame for $H$.
  	\item [3)] There exists a g-Bessel sequence $\{\Gamma_{i}\}_{i\in I}$ such that $Kf=\sum_{i\in I} \Lambda^{\ast}_{i}\Gamma_{i}f$.
  \end{itemize}
 \end{theorem}
\begin{proof}\textcolor{white}{.}
	
it easily obtaind by lemma 2.3, 2.5 and 4.2
\end{proof}

\begin{theorem}\textcolor{white}{.}
	
Let $K_{1}, K_{2}\in End_{\mathcal{A}}^{\ast}(H)$, if $\{\Lambda_{i}\}_{i\in I}$ is an atomic system for $K_{1}$ and $K_{2}$, and $\alpha$,$\beta$ are the scalars, then $\{\Lambda_{i}\}_{i\in I}$ is an atomic system for $\alpha K_{1} + \beta K_{2}$ and $ K_{1}K_{2}$. 
\end{theorem}
\begin{proof}\textcolor{white}{.}
	
Since $\{\Lambda_{i}\}_{i\in I}$ is an atomic system for $K_{1}$ and $K_{2}$ there are positive constants $A_{n} , B_{n} > 0 (n=1,2)$ such that 
\begin{equation}
A_{n}\langle K^{\ast}_{n}f,K^{\ast}_{n}f\rangle\leq \sum_{i\in I} \langle \Lambda_{i}f,\Lambda_{i}f\rangle \leq  B_{n}\langle f,f\rangle , \textcolor{white}{..........}\forall f\in H
\end{equation}
Since
\begin{align*}
\|K^{\ast}_{1}f\|^{2} &=\frac{1}{|\alpha|^{2}}\|\alpha K^{\ast}_{1}f\|^{2} = \frac{1}{|\alpha|^{2}}\|(\alpha K^{\ast}_{1} + \beta K^{\ast}_{2})f - \beta K^{\ast}_{2}f\|^{2}\\
&\geq \frac{1}{|\alpha|^{2}}\|(\alpha K^{\ast}_{1} + \beta K^{\ast}_{2})f\|^{2} - \frac{1}{|\alpha|^{2}}\|\beta K^{\ast}_{2}f\|^{2}
\end{align*}
We have
\begin{align*}
\|(\alpha K^{\ast}_{1} + \beta K^{\ast}_{2})f\|^{2} &\leq |\alpha|^{2}\|K^{\ast}_{1}f\|^{2} + |\beta|^{2}\|K^{\ast}_{2}f\|^{2}\\
&\leq \frac{1}{2}(|\alpha|^{2}\|K^{\ast}_{1}f\|^{2} + |\beta|^{2}\|K^{\ast}_{2}f\|^{2} +  \frac{A_{1}}{A_{2}}|\beta|^{2}\|K^{\ast}_{1}f\|^{2} + \frac{A_{2}}{A_{1}}|\alpha|^{2}\|K^{\ast}_{2}f\|^{2} )\\
&=\frac{A_{2}|\alpha|^{2} + A_{1}|\beta|^{2}}{2A_{1}A_{2}}(A_{1}\|K^{\ast}_{1}f\|^{2} + A_{2}\|K^{\ast}_{2}f\|^{2} )
\end{align*}
Hence 
\begin{equation*}
\sum_{i\in I} \langle \Lambda_{i}f,\Lambda_{i}f\rangle \geq \frac{1}{2}(A_{1}\|K^{\ast}_{1}f\|^{2} + A_{2}\|K^{\ast}_{2}f\|^{2}) \geq \frac{A_{1}A_{2}}{A_{2}|\alpha|^{2} + A_{1}|\beta|^{2}}(\|(\alpha K^{\ast}_{1} + \beta K^{\ast}_{2})f\|^{2})
\end{equation*}
and from inequalities (4.1), we get : 
\begin{equation*}
\sum_{i\in I} \langle \Lambda_{i}f,\Lambda_{i}f\rangle \leq \frac{B_{1} + B_{2}}{2}\|f\|^{2}
, \forall f \in H 
\end{equation*}
therefore, $\{\Lambda_{i}\}_{i\in I}$ is an $(\alpha K_{1} + \beta K_{2} )$-g-frame\\

By theorem 4.3 $\{\Lambda_{i}\}_{i\in I}$ is an atomic system of $(\alpha K_{1} + \beta K_{2} )$
 Now for each $f\in H$, we have 
 \begin{equation*}
\|(K_{1}K_{2})^{\ast}f\|^{2} = \|K^{\ast}_{2}K^{\ast}_{1}f\|^{2}\leq  \|K^{\ast}_{2}\|^{2}\|K^{\ast}_{1}f\|^{2}\
 \end{equation*}
 Hence  $\{\Lambda_{i}\}_{i\in I}$ is an atomic system for $K_{1}$, we have 
 \begin{equation*}
 \frac{A_{1}}{\|K^{\ast}_{2}\|^{2}}\|(K_{1}K_{2})^{\ast}f\|^{2} \leq A_{1}\|K^{\ast}_{1}f\|^{2} \leq \sum_{i\in I} \langle \Lambda_{i}f,\Lambda_{i}f\rangle \leq B_{1}\|f\|^{2} , \forall f\in H
 \end{equation*}
by theorem 4.3 $\{\Lambda_{i}\}_{i\in I}$ is an atomic system for $K_{1}K_{2}$.
\end{proof}
\begin{theorem}\textcolor{white}{.}
	
Let $\{\Lambda_{i}\}_{i\in I}$ and $\{\Gamma_{i}\}_{i\in I}$ be two atomic systems for $K$, and let the corrsponding synthesis operators be $T_{\Lambda}$ and $T_{\Gamma}$ respectively.
If $T_{\Lambda}T^{\ast}_{\Gamma} = 0$ and $U$ or $V$ is surjective satisfying $UK^{\ast}=K^{\ast}U$ or $VK^{\ast}=K^{\ast}V$, then $\{ \Lambda_{i}U + \Gamma_{i}V\}_{i\in I}$ is an atomic system for $K$.
\end{theorem}
\begin{proof}\textcolor{white}{.}
	
Since $\{\Lambda_{i}\}_{i\in I}$ and $\{\Gamma_{i}\}_{i\in I}$ are two atomic systems for $K$, by theorem 4.3, $\{\Lambda_{i}\}_{i\in I}$ and $\{\Gamma_{i}\}_{i\in I}$ are two $K$-g-frames for $H$, and so there exist $B_{1} \geq A_{1} > 0 $ and $B_{2} \geq A_{2} > 0 $ such that :
\begin{equation*}
A_{1}\langle K^{\ast}f,K^{\ast}f\rangle \leq \sum_{i\in I}\langle \Lambda_{i} f,\Lambda_{i} f\rangle \leq B_{1}\langle f,f\rangle  ,\textcolor{white}{............} A_{2}\langle K^{\ast}f,K^{\ast}f\rangle \leq \sum_{i\in I}\langle \Gamma_{i} f,\Gamma_{i} f\rangle \leq B_{1}\langle f,f\rangle 
\end{equation*}
Since $T_{\Lambda}T^{\ast}_{\Gamma} = 0$, for any $f\in H$, we have:
\begin{equation*}
\sum_{i\in I}\Lambda^{\ast}_{i}\Gamma_{i}f = \sum_{i\in I}\Gamma^{\ast}_{i}\Lambda_{i}f=0
\end{equation*}
Therefore, for any $f\in H$, we have :
\begin{align*}
\sum_{i\in I} \|(\Lambda_{i}U + \Gamma_{i}V)f\|^{2} &=\sum_{i\in I}\langle \Lambda_{i}U + \Gamma_{i}Vf,\Lambda_{i}U + \Gamma_{i}Vf\rangle\\
&=\sum_{i\in I}\|\Lambda_{i}Uf\|^{2} + \sum_{i\in I}\|\Gamma_{i}Vf\|^{2} + 2Re\sum_{i\in I}\langle \Lambda^{\ast}_{i}\Gamma_{i}f,Uf\rangle \\
&=\sum_{i\in I}\|\Lambda_{i}Uf\|^{2} + \sum_{i\in I}\|\Gamma_{i}Vf\|^{2}\\
&\leq B_{1}\|Uf\|^{2} + B_{2}\|Vf\|^{2}\\
&\leq  (B_{1}\|U\|^{2} + B_{2}\|V\|^{2})\|f\|^{2}\\
\end{align*}
Without loss of generality, assume that U is surjective; then there exists $C>0$ such that $\langle Uf,Uf\rangle \geq C\langle f,f\rangle$ for any $f\in H$.\\
Since $UK^{\ast}=K^{\ast}U$, we have :
\begin{align*}
\sum_{i\in I} \|(\Lambda_{i}U + \Gamma_{i}V)f\|^{2}&=\sum_{i\in I}\|\Lambda_{i}Uf\|^{2} + \sum_{i\in I}\|\Gamma_{i}Vf\|^{2}\\
&\geq \sum_{i\in I}\|\Lambda_{i}Uf\|^{2}\\ 
&\geq A_{1}\|K^{\ast}Uf\|^{2}=A_{1}\|UK^{\ast}f\|^{2}\\
&\geq A_{1}C\|K^{\ast}f\|^{2}
\end{align*}
So, $\{ \Lambda_{i}U + \Gamma_{i}V \}_{i\in I}$ is a $K$-g-frame and thus an atomic system for $K$ by theorem.4.3. 
\end{proof}
 Let B=0, and we get the following corollory.
 \begin{corollary}\textcolor{white}{.}
 	
 Suppose that $K \in End^{\ast}_{A}(H)$ and that $\{\Lambda_{i}\}_{i\in I}$ is an atomic system for $K$ . If $U$ is surjective and $UK^{\ast}=K^{\ast}U$, then $\{\Lambda_{i}U\}_{i\in I}$ is an atomic system for $K$.
  \end{corollary}
Let $U=V=Id_{H}$, then we obtain the following corollory for a $K$-g-frame.
\begin{corollary}\textcolor{white}{.}
	
Let $\{\Lambda_{i}\}_{i\in I}$ and $\{\Gamma_{i}\}_{i\in I}$ be two parseval $K$-g-frames for $H$, with synthesis operator $T_{\Lambda}$
and $T_{\Gamma}$, respectively. If $T_{\Lambda}T_{\Gamma}^{\ast}=0$ then $\{ \Lambda_{i} + \Gamma_{i}\}_{i\in I}$ is a 2-tight $K$-g-frame for $H$.
\end{corollary}
\begin{theorem}\textcolor{white}{.}
	
Let $\{\Lambda_{i}\}_{i\in I}$ and $\{\Gamma_{i}\}_{i\in I}$ be two atomic system for $K$ and let the corresponding synthesis operators be $T_{1}$ and $T_{2}$, respectively.
If $T_{\Lambda}T_{\Gamma}^{\ast}=0$ and $U_{i} \in  End^{\ast}_{A}(H)$ satisfies $R(T_{i}) \subset R(U^{\ast}_{1}T_{i})$ for $i=1,2$; then $\{ \Lambda_{i}U_{1} + \Gamma_{i}U_{2}\}_{i\in I}$ is an atomic system for $K$.
\end{theorem}
\begin{proof}\textcolor{white}{.}
	
Since $T_{1}T^{\ast}_{2}= 0$, we have :
\begin{align}
\sum_{i\in I}\langle (\Lambda_{i}U_{1} + \Gamma_{i}U_{2})f, (\Lambda_{i}U_{1} + \Gamma_{i}U_{2})f\rangle &= \sum_{i\in I}\langle \Lambda_{i}U_{1}f,\Lambda_{i}U_{1}f\rangle  + \sum_{i\in I}\langle \Gamma_{i}U_{2}f,\Gamma_{i}U_{2}f\rangle\\
&=\langle T^{\ast}_{1}U_{1}f,T^{\ast}_{1}U_{1}f\rangle  + \langle T^{\ast}_{2}U_{2}f,T^{\ast}_{2}U_{2}f\rangle\\
&=\langle (U^{\ast}_{1}T_{1})^{\ast}f, (U^{\ast}_{1}T_{1})^{\ast}f\rangle + \langle (U^{\ast}_{2}T_{2})^{\ast}f, (U^{\ast}_{2}T_{2})^{\ast}f\rangle
\end{align}
Since $\{\Lambda_{i}\}_{i\in I}$ and $\{\Gamma_{i}\}_{i\in I}$ are atomic systems, they are $K$-g-frames by theorem 4.3. Thus from lemma 2.4, we have that $R(K) \subset R(T_{i}) \subset R(U^{\ast}_{i}T_{i})$. So by lemma 2.6, for each $i=1,2$, there existe $\lambda_{i} > 0$ such that :
\begin{equation*}
KK^{\ast} \leq \lambda_{i}(U^{\ast}_{i}T_{i})(U^{\ast}_{i}T_{i})^{\ast}
\end{equation*}
By (4.4) for each $f\in H$, we have 
\begin{align*}
\sum_{i\in I}\langle (\Lambda_{i}U_{1} + \Gamma_{i}U_{2})f, (\Lambda_{i}U_{1} + \Gamma_{i}U_{2})f\rangle &=\sum_{i\in I}\langle (U^{\ast}_{1}T_{i})^{\ast}f,(U^{\ast}_{1}T_{i})^{\ast}f\rangle +  \sum_{i\in I}\langle (U^{\ast}_{2}T_{i})^{\ast}f,(U^{\ast}_{2}T_{i})^{\ast}f\rangle\\
&\geq (\frac{1}{\lambda_{1}} + \frac{1}{\lambda_{2}})\sum_{i\in I}\langle K^{\ast}f,K^{\ast}f\rangle 
\end{align*}
Hence $\{ \Lambda_{i}U + \Gamma_{i}V \}_{i\in I}$ is a $K$-g-frame and thus an atomic system for $K$ by theorem 4.3. 
\end{proof}
Before the following result, we need a simple lemma, wich is a generalization of [19, Theorem 3.5].
\begin{lemma}\textcolor{white}{.}
	
Let $\{\Lambda_{i}\}_{i\in I}$ be a Bessel g-sequence for $H$ with a frame operator $S_{\Lambda}$. Then $\{\Lambda_{i}\}_{i\in I}$ is a $K$-g-frame if and only if there exists $\lambda > 0$ such that $S_{\Lambda} \geq \lambda KK^{\ast}$. 
\end{lemma}
\begin{proof}\textcolor{white}{.}
	
$\{\Lambda_{i}\}_{i\in I}$ is a $K$-g-frame with frame bounds, A,B and a frame operator $S_{\Lambda}$ if and only if 
\begin{equation*}
A \langle K^{\ast}f, K^{\ast}f \rangle \leq \sum_{i\in I} \langle \Lambda_{i}f, \Lambda_{i}f \rangle = \langle S_{\Lambda}f,f\rangle \leq A \langle f,f \rangle, \forall f\in H
\end{equation*}
That is,
\begin{equation*}
\langle AKK^{\ast}f,f\rangle \leq \langle S_{\Lambda}f,f\rangle \leq \langle Bf,f\rangle, \forall f\in H
\end{equation*}
So the conclusion holds.
\end{proof}
\begin{theorem}\textcolor{white}{.}
	
	Let $\{\Lambda_{i}\}_{i\in I}$ be an atomic system for $K$, and let $S_{\Lambda}$ be the frame operator of $\{\Lambda_{i}\}_{i\in I}$. Let $U$ be a positive operator; then $\{ \Lambda_{i} + \Lambda_{i}U\}_{i\in I} $ is an atomic system for $K$. Moreover, for any natural number $n$,  $\{ \Lambda_{i} + \Lambda_{i}U^{n}\}_{i\in I} $ is an atomic system for $K$.
\end{theorem}
\begin{proof}\textcolor{white}{.}
	
Since $\{\Lambda_{i}\}_{i\in I}$ is an atomic system for $K$, by lemma 1.6, $\{\Lambda_{i}\}_{i\in I}$ is a $K$-g-frame for $H$. Then by lemma 4.9 there exists $\lambda > 0$ such that $S_{\Lambda} \geq \lambda KK^{\ast}$.\\
The frame operator for   $\{ \Lambda_{i} + \Lambda_{i}U\}_{i\in I} $ is $(I_{H} + U)^{\ast}S_{\Lambda}(I_{H} + U)$.\\
In fact, for each $f\in H$, we have 
\begin{align*}
\sum_{i\in I}(\Lambda_{i} + \Lambda_{i}U)^{\ast}(\Lambda_{i} + \Lambda_{i}U)f &=(I_{H} + U)^{\ast}\sum_{i\in I}\Lambda^{\ast}_{i}\Lambda_{i}(I_{H}+U)f\\
&= (I_{H} + U)^{\ast}S_{\Lambda}(I_{H} + U)f
\end{align*}
Since $U,S>0$, $(I_{H} + U)^{\ast}S_{\Lambda}(I_{H} + U) \geq S_{\Lambda} \geq \lambda KK^{\ast}$, and again by lemma 4.9, we can conclude that $\{\Lambda_{i} + \Lambda_{i}U\}_{i\in I}$ is a $K$-g-frame and an atomic system for $K$ by theorem 3.3 For any natural number $n$, the frame operator fo  $\{\Lambda_{i} + \Lambda_{i}U^{\ast}\}_{i\in I}$ is $(I_{H} + U^{n})^{\ast}S_{\Lambda}(I_{H} + U^{n})$. Similarly,  $\{\Lambda_{i} + \Lambda_{i}U^{n}\}_{i\in I}$ is an atomic system for $K$
\end{proof}
\subsection*{Acknowledgment}
The authors would like to thank from the anonymous reviewers for carefully reading of the manuscript
and giving useful comments, which will help to improve the paper.

\end{document}